\documentclass[10pt,twoside]{siamart190516}
\usepackage[english]{babel}
\usepackage{graphicx,epstopdf,epsfig}
\usepackage{amsfonts,epsfig,fancyhdr,graphics, hyperref,amsmath,amssymb}
 \usepackage{pgfplots}
\usepackage{graphicx,epstopdf}
\usepackage[caption=false]{subfig}
 \tikzset{
    vertex/.style = {
        circle,
        draw,
        outer sep = 3pt,
        inner sep = 3pt,
    },edge/.style = {->,> = latex'}
}
\usetikzlibrary{arrows}
\newcommand{\HE}{Name of Handling Editor}
\newcommand{\DoS}{Month/Day/Year}
\newcommand{\DoA}{Month/Day/Year}
\newcommand{\CA}{Name of Corresponding Author}

\newcommand{\Names}{Balaji R., Ravindra B. Bapat, and Shivani Goel}
\newcommand{\Title}{Resistance distance in directed cactus graphs}

\newtheorem{thm}[theorem]{Theorem}
\newtheorem{prop}[theorem]{Proposition}

\newtheorem{example}[theorem]{Example}

\newtheorem{lem}[theorem]{Lemma}

\def\adj{\mathop{\rm adj}}

\def\rank{\mathop{\rm rank}}
\newcommand{\rr}{\mathbb{R}}
\newcommand{\1}{\mathbf{1}}

\newcommand{\E}{\mathcal{E}}

\def\det{{\rm det}}
\def\csum{{\rm cofsum}}
\def\z{\mathbf{Z}}

\renewcommand{\theequation}{\thesection.\arabic{equation}}
\renewtheorem{theorem}{Theorem}[section]

\begin{document}


\setcounter{page}{1}

\thispagestyle{empty}

 \title{\Title\thanks{Received
 by the editors on \DoS.
 Accepted for publication on \DoA. 
 Handling Editor: \HE. Corresponding Author: \CA}}

\author{Balaji R. \thanks{Department of Mathematics, IIT Madras, Chennai, India
(\email{balaji5@iitm.ac.in}). Supported by Department of science and Technology -India under the project MATRICS (MTR/2017/000342).}
\and Ravindra B. Bapat \thanks{Theoretical Statistics and Mathematics Unit, Indian Statistical Institute, Delhi, India (\email{rbb@isid.ac.in}).}
\and 
Shivani Goel \thanks{Department of Mathematics, IIT Madras, Chennai, India
(\email{shivani.goel.maths@gmail.com}).}}

\markboth{\Names}{\Title}

\maketitle

\begin{abstract}
Let $G=(V,E)$ be a strongly connected and balanced digraph with vertex set $V=\{1,\dotsc,n\}$.
The classical distance $d_{ij}$ between any two vertices $i$ and $j$ in $G$ is the minimum length of all the 
directed paths joining $i$ and $j$. 
The resistance distance (or, simply the resistance) between any two vertices $i$ and $j$ in $V$ is defined by
$r_{ij}:=l_{ii}^{\dag}+l_{jj}^{\dag}-2l_{ij}^{\dag}$, where $l_{pq}^{\dagger}$ is the $(p,q)^{\rm th}$ entry of the Moore-Penrose inverse of $L$ which is the Laplacian matrix of $G$. In practice, the resistance $r_{ij}$ is
more significant than the classical distance. One reason for this is, numerical examples show that
the resistance distance between $i$ and $j$ is always less than or equal to the classical distance, i.e.
$r_{ij} \leq d_{ij}$. However, no proof for this inequality is known. In this paper, we show that this inequality holds for all directed cactus graphs. 
\end{abstract}

\begin{keywords}
Strongly connected balanced digraph, directed cactus graph, Laplacian matrix, Moore-Penrose inverse, cofactor sums.
\end{keywords}
\begin{AMS}
05C50. 
\end{AMS}



\section{Introduction}

Consider a simple undirected connected graph $H = (W,E)$, where $W:=\{1,\dotsc,n\}$ is the set of all vertices and $E$ is the set of all edges. If $i$ and $j$ are adjacent in $W$, we write $ij$ to denote an element in  $E$ and $\delta_i$ to denote the degree of the vertex $i$. There are several matrices  associated with $H$. Define
\begin{equation*}
    l_{ij}:= \begin{cases}
    \delta_i & \text{if}~ i=j \\ 
    -1 & \text{if}~ i \neq j ~\text{and}~ ij \in E \\
    0 & \text{otherwise}.
    \end{cases}
\end{equation*}
The Laplacian matrix of $H$ is then the matrix $L:=[l_{ij}]$. If $x$ and $y$ are any two vertices, then the classical distance $d_{xy}$ is defined as the length of the shortest path connecting $x$ and $y$. If there are multiple paths connecting two distinct vertices, then in applications, those two vertices are interpreted as better communicated. Thus, it makes more sense to define 
a distance which is shorter than the classical distance. Let $L^\dag$  denote the Moore-Penrose inverse of $L$ and the $(i,j)^{\rm th}$-entry of $L^{\dag}$ be $l_{ij}^\dag$. The resistance distance $R_{xy}$ between vertices $x$ and $y$ is defined by
\begin{equation} \label{res}
  R_{xy} := l_{xx}^\dag+l_{yy}^\dag  - 2l_{xy}^\dag.
\end{equation}
In order to address the drawbacks of classical distance, Klein and Randi\'c  introduced the resistance 
distance \cref{res} in \cite{kr}. A connected graph is a formal representation of an electrical network with unit resistance placed on each of its edges. If $i$ and $j$ are any two vertices, and if current is allowed to enter the electrical circuit only at $i$ and to leave at $j$, then the effective resistance between $i$ and $j$ is same as the resistance distance $R_{ij}$.
The resistance matrix is now defined by $[R_{ij}]$. Resistance matrices of connected graphs have a wide literature. Klein and Randi\'c \cite{kr} showed that  $R_{ij}: W \times W \to \rr$ is a metric. A formula for the inverse of a resistance matrix is obtained in \cite{rbbres}. This in turn extends the remarkable formula of Graham and Lov\'asz to find the inverse of 
the distance matrix of a tree.
All resistance matrices are Euclidean distance matrices (EDMs): see Bapat and Raghavan \cite{Bapat}. Hence the wide theory of EDMs are applicable to resistance matrices. Our interest 
on resistance matrices in this paper is the following inequality:
If $u$ and $v$ are any two distinct vertices, then $R_{uv} \leq d_{uv}$: see
Theorem D in \cite{kr}.

\subsection{Extension of resistance to digraphs}
In \cite{bal_bap_shiv}, the concept of resistance distance is extended for digraphs. Let $G=(V,\E)$ be a simple digraph with vertex set $V = \{1,2,\ldots,n\}$ and edge set $\E$. For $i,j \in V$, we write $(i,j) \in \E$ whenever there is a directed edge from $i$ to $j$. For a vertex $i \in V$, the indegree $\delta_{i}^{in}$ and the outdegree $\delta_{i}^{out}$ are defined as follows:
\begin{equation*}
    \delta_{i}^{in} := |\{j \in V| ~(j,i) \in \E\}| ~~\mbox{and}~~
    \delta_{i}^{out} := |\{j \in V| ~(i,j) \in \E\}|.
\end{equation*}
The Laplacian matrix of $G$ is $L(G) = (l_{ij})$, where for each $i,j \in V$
\begin{equation*}
    l_{ij} := \begin{cases}
    \delta_i^{out} & \text{if}~ i=j \\ 
    -1 & \text{if}~ i \neq j ~\text{and}~ (i,j) \in \E \\
    0 & \text{otherwise}.
    \end{cases}
\end{equation*}
A digraph is strongly connected if there is a directed path between any two distinct vertices.
If $\delta_x^{out}=\delta_x^{in}$, then the vertex $x$ is said to be balanced. A digraph is called balanced if every vertex is balanced. 
In this paper, we consider only strongly connected and balanced digraphs. With this assumption, the Laplacian $L(G)$ will have the following properties:
$\rank({L(G)}) = n-1$, 
row and column sums of $L(G)$ are equal to zero: see \cite{bal_bap_shiv}.  As usual, let $L^\dag = (l_{ij}^\dag)$ be the Moore-Penrose inverse of $L(G)$. It can be noted that $L(G)$ is not a symmetric matrix in general. From now on, we will use $L$ to denote the Laplacian matrix $L(G)$.

The resistance $r_{ij}$ between any two vertices $i$ and $j$ in $V$ is defined by
\begin{equation}\label{r_ij}
    r_{ij}: = l_{ii}^\dag+l_{jj}^\dag-2l_{ij}^\dag.
\end{equation}
In \cite{bal_bap_shiv}, by using certain specialized results on $\z$-matrices and the Moore-Penrose inverse, it is shown that 
\[r_{ij} \geq 0~~\forall i,j, \]
 and for each $i,j,k \in V$ 
\begin{equation*}
    r_{ij} \leq r_{ik}+r_{kj}.
\end{equation*}
For each distinct pair of vertices $i$ and $j$ in $V$, let $d_{ij}$ be the length of the shortest directed path from $i$ to $j$ and define $d_{ii}:=0$. The non-negative real number $d_{ij}$ is the classical distance between $i$ and $j$. By numerical experiments, we noted that the inequality $r_{ij} \leq d_{ij}$ always holds. 

\begin{example}\rm
Consider the graph below. 
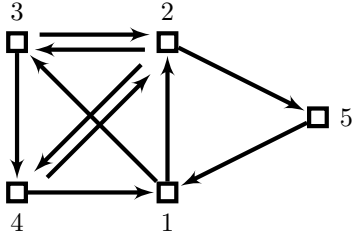
\begin{figure}[tbhp]
\centering
\begin{tikzpicture}[shorten >=1pt, auto, node distance=3cm, ultra thick,
   node_style/.style={circle,draw=black,fill=white !20!,font=\sffamily\Large\bfseries},
   edge_style/.style={draw=black, ultra thick}]
\node[label=above:$3$,draw] (3) at  (0,0) {};
\node[label=above:$2$,draw]  (2) at  (2,0) {};
\node[label=right:$5$,draw] (5) at  (4,-1) {};
\node[label=below:$1$,draw] (1) at  (2,-2) {};
\node[label=below:$4$,draw] (4) at  (0,-2) {};
\draw[edge]  (1) to (2);
\draw[edge]  (1) to (3);
\draw[edge]  (1.7,-0.1) to (0.2,-0.1);
\draw[edge]  (0.3,0.1) to (1.8,0.1);
\draw[edge]  (1.65,-0.3) to (0.2,-1.75);
\draw[edge]  (0.4,-1.8) to (1.8,-0.4);
\draw[edge]  (2) to (5);
\draw[edge]  (3) to (4);
\draw[edge]  (4) to (1);
\draw[edge]  (5) to (1);
\end{tikzpicture}
\caption{A strongly connected and balanced digraph on $5$ vertices.} \label{eg,digraph}
\end{figure}

The Laplacian matrix and its Moore-Penrose inverse are
\begin{equation*}
L = 
\left[
{\begin{array}{rrrrrrr}
2 & -1 & -1 & 0 & 0 \\
0 & 3 & -1 & -1 & -1 \\
0 & -1 & 2 & -1 & 0 \\
-1 & -1 & 0 & 2 & 0 \\
-1 & 0 & 0 & 0 & 1
\end{array}}
\right]~~\mbox{and}~~ L^\dag = 
\left[
{\begin{array}{rrrrrrr}
\frac{9}{35} & 0 & \frac{1}{35} & -\frac{3}{35} & -\frac{1}{5} \\
-\frac{4}{35} & \frac{1}{5} & -\frac{2}{35} & -\frac{1}{35} & 0 \\
-\frac{6}{35} & 0 & \frac{11}{35} & \frac{2}{35} & -\frac{1}{5} \\
-\frac{1}{35} & 0 & -\frac{4}{35} & \frac{12}{35} & -\frac{1}{5} \\
\frac{2}{35} & -\frac{1}{5} & -\frac{6}{35} & -\frac{2}{7} & \frac{3}{5}
\end{array}}
\right].
\end{equation*}
The resistance and distance matrices of $G$ are:
\begin{equation*}
R =[r_{ij}] = \left[
{\begin{array}{rrrrrrr}
0 & \frac{16}{35} & \frac{18}{35} & \frac{27}{35} & \frac{44}{35} \\
\frac{24}{35} & 0 & \frac{22}{35} & \frac{3}{5} & \frac{4}{5} \\
\frac{32}{35} & \frac{18}{35} & 0 & \frac{19}{35} & \frac{46}{35} \\
\frac{23}{35} & \frac{19}{35} & \frac{31}{35} & 0 & \frac{47}{35} \\
\frac{26}{35} & \frac{6}{5} & \frac{44}{35} & \frac{53}{35} & 0
\end{array}}
\right]~~\mbox{and}~~
D = \left[
{\begin{array}{rrrrrrr}
0 & 1 & 1 & 2 & 2 \\
2 & 0 & 1 & 1 & 1 \\
2 & 1 & 0 & 1 & 2 \\
1 & 1 & 2 & 0 & 2 \\
1 & 2 & 2 & 3 & 0
\end{array}}
\right].
\end{equation*}
\end{example}

It is easily seen that $r_{ij} \leq d_{ij}$ for each $i,j$.
Given a general strongly connected and balanced digraph, we do not know how to prove the above inequality. In this paper, when $G$ is a directed cactus graph, we give a proof for this inequality. 

\subsection{Directed cactus graphs}
\begin{definition}\label{cactus}\rm
A directed cactus graph is a strongly connected digraph in which each edge is contained in exactly one directed cycle. 
\end{definition}
Here is an equivalent condition for a directed cactus: A digraph $G$ is a directed cactus if and only if any two directed cycles of $G$ share at most one common vertex.  In a directed cactus, for each vertex $i$,  $\delta_{i}^{in}=\delta_{i}^{out}$ and hence balanced. 
The graph $G$ given in \cref{eg:cactus} is a directed cactus graph.
\begin{figure}[tbhp]
\centering
\begin{tikzpicture}[shorten >=1pt, auto, node distance=3cm, ultra thick,
   node_style/.style={circle,draw=black,fill=white !20!,font=\sffamily\Large\bfseries},
   edge_style/.style={draw=black, ultra thick}]
\node[label=above:$1$,draw] (1) at  (1.75,0) {};
\node[label=above:$2$,draw] (2) at  (3.5,0) {};
\node[label=below:$3$,draw] (3) at  (3.5,-2) {};
\node[label=above:$6$,draw] (6) at  (0,0) {};
\node[label=below:$4$,draw] (4) at  (0,-2) {};
\node[label=left:$5$,draw] (5) at  (-1.5,-1) {};
\node[label=below:$7$,draw] (7) at  (1.75,-2) {};
\draw[edge]  (1) to (2);
\draw[edge]  (2) to (3);
\draw[edge]  (3) to (1);
\draw[edge]  (1) to (4);
\draw[edge]  (4) to (5);
\draw[edge]  (5) to (6);
\draw[edge]  (6) to (1);
\draw[edge]  (1.65,-0.38) to (1.65,-1.7);
\draw[edge]  (1.85,-1.62) to (1.85,-0.35);
\end{tikzpicture}
\caption{A directed cactus graph on $7$ vertices.} \label{eg:cactus}
\end{figure}
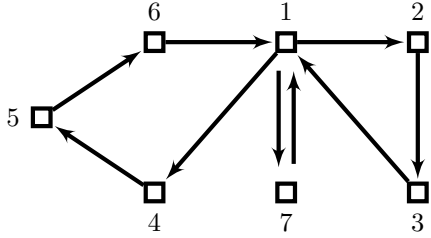

Distance matrices of directed cactus appear in \cite{chen}. An interesting formula for the determinant of the distance matrix $D:=(d_{ij})$ of a cactoid graph and an expression for the inverse of $D$ are computed in \cite{chen}. 

\section{Preliminaries}

\begin{definition}\rm\label{directedcycle}
A directed cycle graph is a directed version of a cycle graph with all edges being oriented in the same direction. For $n>1$, we shall use $C_n = (V,\E)$ to denote a directed cycle on $n$ vertices. 
\end{definition}
An example of a directed cycle on $5$ vertices is shown in \cref{eg:dircycle}.
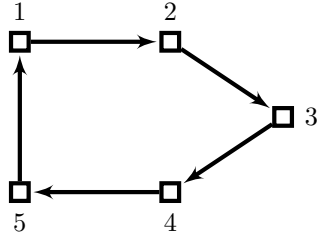
\begin{figure}[tbhp]
\centering
\begin{tikzpicture}[shorten >=1pt, auto, node distance=3cm, ultra thick,
   node_style/.style={circle,draw=black,fill=white !20!,font=\sffamily\Large\bfseries},
   edge_style/.style={draw=black, ultra thick}]
\node[label=above:$1$,draw] (1) at  (0,0) {};
\node[label=above:$2$,draw] (2) at  (2,0) {};
\node[label=right:$3$,draw] (3) at  (3.5,-1) {};
\node[label=below:$4$,draw] (4) at  (2,-2) {};
\node[label=below:$5$,draw] (5) at  (0,-2) {};
\draw[edge]  (1) to (2);
\draw[edge]  (2) to (3);
\draw[edge]  (3) to (4);
\draw[edge]  (4) to (5);
\draw[edge]  (5) to (1); 
\end{tikzpicture}
\caption{Directed cycle Graph $C_5$.} \label{eg:dircycle}
\end{figure}

\begin{definition}\rm
Suppose $G = (V,\E)$ is a digraph with vertex set $V = \{1,2,...,n\}$ and Laplacian matrix $L$. A spanning \emph{tree} of $G$ rooted at vertex $i$ is a connected subgraph $T$ with vertex set $V$ such that
\begin{enumerate}
\item [\rm (i)] Every vertex of $T$ other than $i$ has indegree $1$.
\item [\rm (ii)] The vertex $i$ has indegree $0$.
\item [\rm (iii)] $T$ has no directed cycles. 
\end{enumerate}
\end{definition}

\begin{example} \rm
The graph $H$ in \cref{eg:span_tree} has two spanning trees rooted at $1$.
\begin{figure}[tbhp]
\centering
~~~\subfloat[]{\begin{tikzpicture}[shorten >=1pt, auto, node distance=3cm, ultra thick,
   node_style/.style={circle,draw=black,fill=white !20!,font=\sffamily\Large\bfseries},
   edge_style/.style={draw=black, ultra thick}]
\node[label=above:$1$,draw] (1) at  (0,0) {};
\node[label=above:$2$,draw] (2) at  (2,0) {};
\node[label=below:$3$,draw] (3) at  (3.5,-1) {};
\node[label=below:$4$,draw] (4) at  (2,-2) {};
\node[label=below:$5$,draw] (5) at  (0,-2) {};
\draw[edge]  (1) to (2);
\draw[edge]  (2) to (3);
\draw[edge]  (3) to (4);
\draw[edge]  (4) to (5);
\draw[edge]  (-0.1,-1.6) to (-0.1,-0.3); 
\draw[edge]  (0.1,-0.35) to (0.1,-1.7); 
\end{tikzpicture}}
~~~~~~~~~~\subfloat[]{\begin{tikzpicture}[shorten >=1pt, auto, node distance=3cm, ultra thick,
   node_style/.style={circle,draw=black,fill=white !20!,font=\sffamily\Large\bfseries},
   edge_style/.style={draw=black, ultra thick}]
\node[label=above:$1$,draw] (1) at  (0,0) {};
\node[label=above:$2$,draw] (2) at  (2,0) {};
\node[label=below:$3$,draw] (3) at  (3.5,-1) {};
\node[label=below:$4$,draw] (4) at  (2,-2) {};
\node[label=below:$5$,draw] (5) at  (0,-2) {};
\draw[edge]  (1) to (2);
\draw[edge]  (2) to (3);
\draw[edge]  (3) to (4);
\draw[edge]  (4) to (5);
\end{tikzpicture} 
~\begin{tikzpicture}[shorten >=1pt, auto, node distance=3cm, ultra thick,
   node_style/.style={circle,draw=black,fill=white !20!,font=\sffamily\Large\bfseries},
   edge_style/.style={draw=black, ultra thick}]
\node[label=above:$1$,draw] (1) at  (0,0) {};
\node[label=above:$2$,draw] (2) at  (2,0) {};
\node[label=below:$3$,draw] (3) at  (3.5,-1) {};
\node[label=below:$4$,draw] (4) at  (2,-2) {};
\node[label=below:$5$,draw] (5) at  (0,-2) {};
\draw[edge]  (1) to (2);
\draw[edge]  (2) to (3);
\draw[edge]  (3) to (4);
\draw[edge]  (1) to (5); 
\end{tikzpicture}}
\caption{(a) Digraph $H$ (b) Spanning trees of $H$ rooted at $1$.} \label{eg:span_tree}
\end{figure}
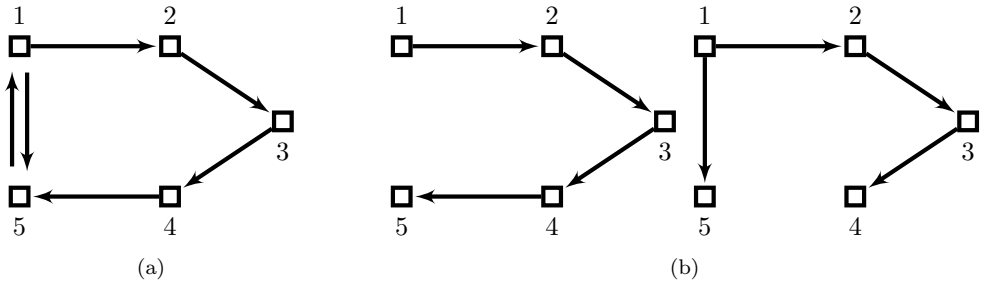 
\end{example}

We use the following notation.
If $\Delta_1$ and $\Delta_2$ are non-empty subsets of $\{1,\dotsc,n\}$ and
$\pi:\Delta_1 \rightarrow \Delta_2$ is a bijection, then the pair $\{i,j\} \subset \Delta_1$ is called an inversion in $\pi$ if $i<j$ and $\pi(i)>\pi(j)$. Let $n(\pi)$ be the number of inversions in $\pi$.
For an $n \times n$ matrix $A$, $A[\Delta_1,\Delta_2]$ will denote the submatrix of $A$ obtained by choosing rows and columns corresponding to $\Delta_1$ and $\Delta_2$, respectively. For $\Delta \subseteq \{1,2,\ldots,n\}$, we define $\alpha(\Delta) = \sum_{i \in \Delta} i$.
Our main tool will be the following theorem from \cite{chaiken}. 

\begin{thm}[All minors matrix tree theorem]\label{mtt} Let $G=(V,E)$ be a digraph with vertex set $V = \{1,2,\ldots,n\}$ and Laplacian matrix $L$. 
Let $\Delta_1,\Delta_2 \subset V$ be such that $|\Delta_1|=|\Delta_2|$. Then
\[ \det(L[\Delta_1^c,\Delta_2^c]) = (-1)^{\alpha(\Delta_1)+ \alpha(\Delta_2)} \sum_F (-1)^{n(\pi)}.\]
where the sum is over all spanning forests $F$ such that
\begin{enumerate}
    \item[\rm(a)] $F$ contains exactly $|\Delta_1|=|\Delta_2|$ trees.
    \item[\rm(b)] each tree in $F$ contains exactly one vertex in $\Delta_2$ and exactly one vertex in $\Delta_1$.
    \item[\rm(c)] each directed edge in $F$ is directed away from the vertex in $\Delta_2$ of the tree containing that directed edge. (i.e. each vertex in $\Delta_2$ is the root of the tree containing it.)
\end{enumerate}
$F$ defines a bijection $\pi:\Delta_1 \rightarrow \Delta_2$ such that $\pi(j)=i$ if and only if $i$ and $j$ are in the same oriented tree of $F$.
\end{thm}  
Let $\kappa(G,i)$ be the number of spanning trees of $G$ rooted at $i$. By \cref{mtt}, it immediately follows that 
\begin{equation}\label{kappa}
       \kappa(G,i) = \det(L[\{i\}^c,\{i\}^c]).
\end{equation}

Let $i,j,k \in V$. We introduce two notation.
\begin{enumerate}
\item Let $\#(F[\{i\rightarrow\},\{j\rightarrow\}])$ denote the number of spanning forests $F$ of $G$ such that 
(i) $F$ contains exactly $2$ trees, (ii) each tree in $F$ contains either $i$ or $j$, and (iii) vertices $i$ and $j$ are the roots of the respective trees containing them. 
\item Let $\#(F[\{k \rightarrow\},\{j\rightarrow,i\}])$ denote the number of spanning forests $F$ of $G$ such that (i) $F$ contains exactly $2$ trees, (ii) each tree in $F$ exactly contains either $k$ or both $i$ and $j$, and (iii) vertices $k$ and $j$ are the roots of the respective trees containing them. 
\end{enumerate}
From \cref{mtt}, we deduce the following proposition which in turn will be used to prove our main result.
\begin{prop}\label{mtt:2vert:rem} 
Let $i,j \in V$ be two distinct vertices. Then
\begin{enumerate}
\item[\rm(a)]  \[\det(L[\{i,j\}^c,\{i,j\}^c]) = \#(F[\{i\rightarrow\},\{j\rightarrow\}]).\]
\item[\rm(b)] If $i \neq n$ and $j \neq n$, then
\[\det(L[\{n,i\}^c,\{n,j\}^c]) = (-1)^{i+j} \#(F[\{n\rightarrow\},\{j\rightarrow,i\}]).\]
\item[\rm(c)] If $i \neq 1$ and $j \neq 1$, then
\[\det(L[\{1,i\}^c,\{1,j\}^c]) = (-1)^{i+j} \#(F[\{1\rightarrow\},\{j\rightarrow,i\}]).\]
\end{enumerate}
\end{prop}

\begin{proof}
Substituting $\Delta_1 = \Delta_2 = \{i,j\}$ in \cref{mtt}, we have
\begin{equation}\label{Emtt:ij:rem}
    \det(L[\{i,j\}^c,\{i,j\}^c]) = (-1)^{2i+2j} \sum_F (-1)^{n(\pi)}
\end{equation}
where the sum is over all forests $F$ such that (i) $F$ contains exactly $2$ trees, (ii) each tree in $F$ contains either $i$ or $j$, and (iii) vertices $i$ and $j$ are the roots of the respective trees containing them. Since for each such forest $F$, $\pi(i) = i$ and $\pi(j) = j$, there are no inversions in $\pi$. Thus $n(\pi) = 0$. Hence from \cref{Emtt:ij:rem}, we have 
\[\det(L[\{i,j\}^c,\{i,j\}^c]) = \#(F[\{i\rightarrow\},\{j\rightarrow\}]).\]
This completes the proof of (a). 

To prove (b), we substitute $\Delta_1 = \{n,i\}$ and $\Delta_2=\{n,j\}$ in \cref{mtt} to obtain
\begin{equation}\label{Emtt:ijn:rem}
 \det(L[\{n,i\}^c,\{n,j\}^c]) =  (-1)^{2n+i+j} \sum_F (-1)^{n(\pi)}
\end{equation}
where the sum is over all forests $F$ such that (i) $F$ contains exactly $2$ trees, (ii) each tree in $F$ exactly contains either $n$ or both $i$ and $j$, and (iii) vertices $n$ and $j$ are the roots of the respective trees containing them. For each such forest $F$, $\pi(n) = n$ and $\pi(i) = j$. Since $i,j <n$, there are no inversions in $\pi$ and so $n(\pi) = 0$. From \cref{Emtt:ijn:rem}, we have
\[\det(L[\{n,i\}^c,\{n,j\}^c]) = (-1)^{i+j} \#(F[\{n\rightarrow\},\{j\rightarrow,i\}]).\]
Hence (b) is proved. The proof of (c) is similar to the proof of (b).
\end{proof}

Since $\rank(L)=n-1$ and $L \1=L' \1=0$, all the cofactors of $L$ are equal. Hence, $\kappa(G,i)$ is independent of $i$. From here on, we shall denote $\kappa(G,i)$ simply by $\kappa(G)$.  
 The following is one of the main results in \cite{bal_bap_shiv}.
\begin{thm}\label{rcofsum}
For every distinct $i,j \in V$,
\[r_{ij}+r_{ji}  = \frac{2}{\kappa(G)}\det(L[\{i,j\}^c,\{i,j\}^c])\]
where $r_{ij}$ is defined in $({r_ij})$.
\end{thm} 
The following lemma can be verified by direct computation and appears in \cite{bal_bap_shiv}.
\begin{lem}\label{pinv}
Let $L$ be a $\z$-matrix such that $L\1 = L' \1 = 0$ and $\rank (L) = n-1$. If $e$ is the vector of all ones in $\mathbb{R}^{n-1}$, then $L$ can be partitioned as
\[L =\left[
\begin{array}{cccc}
B & -B e \\
-e' B  & e'B e \\
\end{array}
\right],\]
where $B$ is a square matrix of order $n-1$ and 
\[L^{\dag} = \left[
\begin{array}{cccc}
B^{-1} - \frac{1}{\displaystyle n} e e' B^{-1} - \frac{1}{\displaystyle n} B^{-1} ee' & - \frac{1}{\displaystyle n} B^{-1} e \\ \\
-\frac{1}{\displaystyle n} e' B^{-1} & 0 \\ 
\end{array}
\right] + \frac{e' B^{-1}e}{\displaystyle n^2} \1\1'.\]
\end{lem}

\begin{lem}\label{L:mdl}
 Let $A$ be an $n \times n$ matrix. If $u$ and $v$ belong to $\rr^n$, then 
 $$\det(A+uv') = \det(A)+v'\adj(A)u.$$
\end{lem}
\begin{proof}
 See Lemma $1.1$ in \cite{jiu}.
\end{proof}

\section{Result}
We now prove our main result.
Let $G = (V,E)$ be a strongly connected and balanced digraph with vertex set $V = \{1,2,\ldots,n\}$, Laplacian matrix $L$ and resistance matrix $R = (r_{ij})$. First, we prove some  lemmas which will be used later.

\begin{lem}\label{Eressumleq:forestsleqtrees}
Let $i,j \in V$. If $(i,j) \in E$ or $(j,i) \in E$, then 
\begin{equation*}
\det(L[\{i,j\}^c,\{i,j\}^c]) \leq \kappa(G).
\end{equation*}
\end{lem}
\begin{proof}
Let $i,j \in V$. Without loss of generality, assume that $(i,j) \in E$. By \cref{mtt:2vert:rem}(a),  $\det(L[\{i,j\}^c,\{i,j\}^c])$  is equal to the number of spanning forests $F$ of $G$ such that (i) $F$ contains exactly $2$ trees, (ii) each tree in $F$ contains either $i$ or $j$, and (iii) vertices $i$ and $j$ are the roots of the respective trees containing them. Let $F$ be one such forest. Now, $F+(i,j)$ is a spanning tree of $G$ rooted at $i$. Moreover, each such spanning forest will give a unique spanning tree rooted at $i$. Since $\kappa(G)$ is the number of spanning trees of $G$ rooted at $i$, 
\[\det(L[\{i,j\}^c,\{i,j\}^c]) \leq \kappa(G). \]
\end{proof}
As $G$ is balanced, we know that $\delta_i^{in} = \delta_i^{out}$ for any $i$. This common value will be called the degree of $i$.

\begin{lem}\label{resleq1fordeg1}
Let $(i,j) \in E$. If either $i$ or $j$ has degree $1$, then $r_{ij} \leq 1$.
\end{lem}
\begin{proof}
Without loss of generality, let $i=1$ and $j=n$. From \cref{pinv}, we have 
\begin{equation}\label{Eresleq1fordeg1pinv}
L^{\dag} = \left[
\begin{array}{cccc}
B^{-1} - \frac{1}{\displaystyle n} e e' B^{-1} - \frac{1}{\displaystyle n} B^{-1} ee' & - \frac{1}{\displaystyle n} B^{-1} e \\ \\
-\frac{1}{\displaystyle n} e' B^{-1} & 0 \\ 
\end{array}
\right] + \frac{e' B^{-1}e}{\displaystyle n^2} \1\1',   
\end{equation}
where $B = \det(L[\{n\}^c,\{n\}^c])$. Let $C = B^{-1}$, $C = (c_{ij})$, $x = Ce$ and $y =C'e$. By a well-known result on $\z$-matrices (Theorem 2.3 in \cite{horn}), we note that $C$ is a non-negative matrix. Using \cref{Eresleq1fordeg1pinv}, we have
\begin{equation}\label{Eresleq1fordeg1:r1n}
\begin{aligned}
    r_{1n} &= l^{\dag}_{11}+l^{\dag}_{nn}-2l^{\dag}_{1n} \\ &= c_{11}-\frac{1}{n}y_1-\frac{1}{n}x_1+\frac{2}{n}x_1 \\ &= c_{11}-\frac{1}{n}(y_1-x_1).
\end{aligned}
\end{equation}
We claim that $x_1 \leq y_1$. To see this, we consider the following cases:
\begin{enumerate}
    \item[(i)] Suppose degree of vertex $1$ is one. For $k\in \{2,3,\ldots,n-1\}$,
    \begin{equation}\label{Eresleq1fordeg1:deg1=1:c1k=det}
    \begin{aligned}
        c_{1k} &= \frac{(-1)^{1+k}}{\det(B)} \det(B[\{k\}^c,\{1\}^c]) \\ &= \frac{(-1)^{1+k}}{\det(L[\{n\}^c,\{n\}^c])} \det(L[\{n,k\}^c,\{n,1\}^c]).
    \end{aligned}
    \end{equation}
Using \cref{kappa} and \cref{mtt:2vert:rem}(b) in \cref{Eresleq1fordeg1:deg1=1:c1k=det}, we get
\begin{equation}\label{Eresleq1fordeg1:deg1=1:c1k=n(F)}
     c_{1k} = \frac{\#(F[\{n\rightarrow\},\{1\rightarrow,k\}])}{\kappa(G)},
\end{equation}
where $\#(F[\{n \rightarrow\},\{1\rightarrow,k\}])$ is the number of spanning forests $F$ of $G$ such that (i) $F$ contains exactly $2$ trees, (ii) each tree in $F$ exactly contains either $n$ or both $1$ and $k$, and (iii) vertices $n$ and $1$ are the roots of the respective trees containing them. As degree of vertex $1$ is one, $(1,n)$ is the only edge directed away from $1$. So, it is not possible for a forest to have a tree such that the tree does not contain the vertex $n$ but contains both the vertices $1$ and $k$ with $1$ as the root. Therefore, no such forest $F$ exists and hence by \cref{Eresleq1fordeg1:deg1=1:c1k=n(F)}, $c_{1k} = 0$ for each $k \in \{2,3,\ldots,n-1\}$. Using the fact that $C$ is a non-negative matrix, we have
\begin{equation}
    x_1 = \sum_{k=1}^{n-1} c_{1k} = c_{11} \leq \sum_{k=1}^{n-1} c_{k1} = y_1.
\end{equation}
Hence $x_1 \leq y_1$.

\item[(ii)] Suppose degree of $n$ is one. Let $e_i$ be the vector $(0,\dotsc,1,\dotsc,0)' \in {\rr}^{n-1}$ with $1$ as its $i^{th}$ coordinate. Then the Laplacian matrix $L$ can be partitioned as
    \begin{equation}\label{}
L = \left[
\begin{array}{cccc}
B & -e_1\\ \\
-e_p' & 1 \\ 
\end{array}
\right],    
\end{equation}
for some $p \in \{1,2,\ldots,n-1\}$. Let $\widetilde{B} = B-E_{p}$, where $E_{p}$ is the 
$(n-1) \times (n-1)$ matrix with $p^{\rm th}$ column equal to $e_1$ and the remaining columns equal to zero.
   It can be seen that $\widetilde{B}$ has all row and column sums zero and so all its cofactors are identical. For $k \in \{2,3,\ldots,n-1\}$, we have
\begin{equation}\label{Eresleq1fordeg1:degn=1:c_k1}
   \begin{aligned}
        c_{k1} &= \frac{(-1)^{1+k}}{\kappa(G)} \det(B[\{1\}^c,\{k\}^c]) \\ &= \frac{(-1)^{1+k}}{\kappa(G)} \det\big((\widetilde{B}+E_{p})[\{1\}^c,\{k\}^c]\big) \\ &= \frac{(-1)^{1+k}}{\kappa(G)} \det(\widetilde{B}[\{1\}^c,\{k\}^c]) \\ &= \frac{1}{\kappa(G)}\frac{\csum(\widetilde{B})}{(n-1)^2}.
    \end{aligned}
\end{equation}
If $p=1$, then for each $k \in \{2,3,\ldots,n-1\}$,
\begin{equation}\label{Eresleq1fordeg1:degn=1:p=1:c_1k}
   \begin{aligned}
        c_{1k} &= \frac{(-1)^{1+k}}{\kappa(G)} \det(B[\{k\}^c,\{1\}^c]) \\ &= \frac{(-1)^{1+k}}{\kappa(G)} \det\big((\widetilde{B}+E_{1})[\{k\}^c,\{1\}^c]\big) \\ &= \frac{(-1)^{1+k}}{\kappa(G)} \det(\widetilde{B}[\{k\}^c,\{1\}^c]) \\ &= \frac{1}{\kappa(G)}\frac{\csum(\widetilde{B})}{(n-1)^2}.
    \end{aligned}
\end{equation}
Therefore from \cref{Eresleq1fordeg1:degn=1:c_k1} and \cref{Eresleq1fordeg1:degn=1:p=1:c_1k}, $c_{k1} = c_{1k}$ for each $k \in \{2,3,\ldots,n-1\}$. Thus,
\begin{equation}
    x_1 = c_{11}+\sum_{k=2}^{n-1} c_{1k} = c_{11}+\sum_{k=2}^{n-1} c_{k1} = y_1.
\end{equation}

If $p\neq 1$, then for each $k \in \{2,3,\ldots,n-1\}$, we have
\begin{equation}\label{Eresleq1fordeg1:degn=1:p=not1:c_1k}
   \begin{aligned}
        c_{1k} &= \frac{(-1)^{1+k}}{\kappa(G)} \det(B[\{k\}^c,\{1\}^c]) \\ &= \frac{(-1)^{1+k}}{\kappa(G)} \det\big((\widetilde{B}+E_{p})[\{k\}^c,\{1\}^c]\big) \\ &= \frac{(-1)^{1+k}}{\kappa(G)} \det\big(\widetilde{B}[\{k\}^c,\{1\}^c]+E_{p}[\{k\}^c,\{1\}^c]\big). 
    \end{aligned}
\end{equation}  
Let $u_\nu$ be the vector $(0,\dotsc,0,1,0,\ldots,0)' \in {\rr}^{n-2}$ with $1$ as its $\nu^{th}$ coordinate. From \cref{Eresleq1fordeg1:degn=1:p=not1:c_1k} and \cref{L:mdl}, we have
\begin{equation}\label{Eresleq1fordeg1:degn=1:p=not1:c_1k=det}
   \begin{aligned}
        c_{1k} &= \frac{(-1)^{1+k}}{\kappa(G)} \det\big(\widetilde{B}[\{k\}^c,\{1\}^c]+ u_1 u_{p-1}'\big) \\ &=  \frac{(-1)^{1+k}}{\kappa(G)} \big(\det(\widetilde{B}[\{k\}^c,\{1\}^c])+u_{p-1}'\adj(\widetilde{B}[\{k\}^c,\{1\}^c])u_1\big) \\ &= \frac{1}{\kappa(G)}\frac{\csum(\widetilde{B})}{(n-1)^2}+\frac{(-1)^{1+k}}{\kappa(G)} (\adj(\widetilde{B}[\{k\}^c,\{1\}^c]))_{p-1,1} \\ &= \frac{1}{\kappa(G)}\frac{\csum(\widetilde{B})}{(n-1)^2}+\frac{(-1)^{1+k+p}}{\kappa(G)} \det(\widetilde{B}[\{1,k\}^c,\{1,p\}^c]).
    \end{aligned}
\end{equation}
Using \cref{mtt:2vert:rem}(c) in \cref{Eresleq1fordeg1:degn=1:p=not1:c_1k=det}, we get
\begin{equation}
  \begin{aligned}\label{Eresleq1fordeg1:degn=1:p=not1:c_1kleq}
     c_{1k} &= \frac{1}{\kappa(G)}\frac{\csum(\widetilde{B})}{(n-1)^2}+\frac{(-1)^{3+2k+2p}}{\kappa(G)} \#(F[\{1\rightarrow\},\{p\rightarrow,k\}]) \\ &= \frac{1}{\kappa(G)}\frac{\csum(\widetilde{B})}{(n-1)^2}-\frac{1}{\kappa(G)} \#(F[\{1\rightarrow\},\{p\rightarrow,k\}]) \\ &\leq  \frac{1}{\kappa(G)}\frac{\csum(\widetilde{B})}{(n-1)^2}. 
         \end{aligned}
\end{equation}
From \cref{Eresleq1fordeg1:degn=1:c_k1} and \cref{Eresleq1fordeg1:degn=1:p=not1:c_1kleq}, we get $c_{1k} \leq c_{k1}$ for each $k \in \{2,3,\ldots,n-1\}$. Thus
\begin{equation}
     x_1 = c_{11}+\sum_{k=2}^{n-1} c_{1k} \leq c_{11}+\sum_{k=2}^{n-1} c_{k1} = y_1.
\end{equation}
So, $x_1 \leq y_1$. 
\end{enumerate}
This proves our claim.
In view of \cref{Eresleq1fordeg1:r1n}, we now obtain
\begin{equation}\label{Eresleq1fordeg1:r1nleqc11}
    r_{1n} \leq c_{11} = \frac{\det(L[\{1,n\}^c,\{1,n\}^c])}{\kappa(G)}.
\end{equation}
By \cref{Eressumleq:forestsleqtrees}, it follows that $r_{1n} \leq 1$. The proof is complete.
\end{proof}

\begin{lem}\label{unique_path:cactus}
Let $G= (V,E)$ be a directed cactus graph on $n$ vertices. Then there is a unique directed path from $i$ to $j$.
\end{lem}

\begin{proof}
Let $i,j \in V$. Since $G$ is strongly connected there exists a directed path from $i$ to $j$. Let $P: i \rightarrow v_1 \rightarrow v_2 \rightarrow \cdots \rightarrow v_{m} \rightarrow j$ be one such path. If possible, let $Q: i \rightarrow w_1 \rightarrow w_2 \rightarrow \cdots \rightarrow w_{l} \rightarrow j$ be another directed path from $i$ to $j$. 

First, we assume all the internal vertices of $P$ and $Q$ are distinct. Since $G$ is strongly connected, there will be a path from $j$ to $i$. Let $R: j \rightarrow w_1' \rightarrow w_2' \rightarrow \cdots \rightarrow w_{\alpha}' \rightarrow i$ be a directed path from $j$ to $i$. If $R$ has no internal vertex in common with $P$ and $Q$, then each edge of $R$ will become a part of two distinct cycles $j \rightarrow w_1' \rightarrow \cdots \rightarrow w_{\alpha}' \rightarrow i \rightarrow v_1 \rightarrow \cdots \rightarrow v_{m} \rightarrow j$ and $j \rightarrow w_1' \rightarrow \cdots \rightarrow w_{\alpha}' \rightarrow i \rightarrow w_1  \rightarrow \cdots \rightarrow w_{l} \rightarrow j$ (see \cref{F:unique_path:cactus}(a)). This is a contradiction to the assumption that $G$ is a directed cactus graph. Suppose $R$ has some internal vertices in common with $P$ and $Q$. Let $w_s'$ and $w_{s'}'$ be the vertices in $R$ such that no vertex of $R$ after $w_s'$ is a vertex of $P$ and no vertex of $R$ after $w_{s'}'$ is a vertex of $Q$. Without loss of generality, we assume $s'>s$. This makes each edge of the path $w_{s'}' \rightarrow w_{s'+1}' \rightarrow \cdots \rightarrow w_{\alpha}' \rightarrow i$ to be a part of two distinct cycles $w_{s'}' \rightarrow w_{s'+1}' \rightarrow \cdots \rightarrow w_{\alpha}' \rightarrow i \rightarrow v_1 \rightarrow \cdots \rightarrow w_s' \rightarrow \cdots \rightarrow w_{s'}'$ and $w_{s'}' \rightarrow w_{s'+1}' \rightarrow \cdots \rightarrow w_{\alpha}' \rightarrow i \rightarrow w_1 \rightarrow \cdots \rightarrow w_{s'}'$ (see \cref{F:unique_path:cactus}(b)), which is a contradiction. 

\begin{figure}[tbhp]
\centering
\subfloat[]{\begin{tikzpicture}[shorten >=1pt, auto, node distance=3cm, ultra thick,
   node_style/.style={circle,draw=black,fill=white !20!,font=\sffamily\Large\bfseries},
   edge_style/.style={draw=black, ultra thick}]
\node[label=left:$i$,draw] (i) at  (-1,-1) {};
\node[label=above:$P$] at  (0.25,-0.25) {};
\node[label=above:$v_1$,draw] (v_1) at  (1.5,0.5) {};
\node[label=above:$v_2$,draw] (v_2) at  (4,0.5) {};
\node[label=above:$v_{m}$,draw] (v_{m-1}) at  (6.5,0.5) {};
\node[label=right:$j$,draw] (j) at  (9,-1) {};

\node[label=below:$w_1$,draw] (w_1) at  (1.5,-2.5) {};
\node[label=below:$w_2$,draw] (w_2) at  (4,-2.5) {};
\node[label=below:$Q$] at  (0.25,-1.75) {};
\node[label=below:$w_{l}$,draw] (w_{l-1}) at  (6.5,-2.5) {};

\node[label=below:$w_1'$,draw] (w_1') at  (6.5,-1) {};
\node[label=above:$R$] at (7.75,-1.15) {}; 
\node[label=below:$w_2'$,draw] (w_2') at  (4,-1) {};
\node[label=below:$w_{\alpha}'$,draw] (w_{alpha-1}') at  (1.5,-1) {};

\draw[edge]  (i) to (v_1);
\draw[edge]  (v_1) to (v_2);
\draw[edge]  (v_{m-1}) to (j);
\draw[edge]  (i) to (w_1);
\draw[edge]  (w_1) to (w_2);
\draw[edge]  (w_{l-1}) to (j);

\draw[edge]  (w_{alpha-1}') to (i);
\draw[edge]  (w_1') to (w_2');
\draw[edge]  (j) to (w_1');

\begin{scope}[dashed]
\draw[edge]  (v_2) to (v_{m-1});
\draw[edge]  (w_2) to (w_{l-1});
\draw[edge]  (w_2') to (w_{alpha-1}');
\end{scope}
\end{tikzpicture}}

\subfloat[]{\begin{tikzpicture}[shorten >=1pt, auto, node distance=3cm, ultra thick,
   node_style/.style={circle,draw=black,fill=white !20!,font=\sffamily\Large\bfseries},
   edge_style/.style={draw=black, ultra thick}]
\node[label=left:$i$,draw] (i) at  (-1,-1) {};
\node[label=above:$P$] at  (0,-0.25) {};
\node[label=above:$v_1$,draw] (v_1) at  (1,0.5) {};
\node[label=above:$v_2$,draw] (v_2) at  (3,0.5) {};
\node[label=above:$w_s'$,draw] (w_s'clone) at  (5,0.5) {};
\node[label=above:$v_{m}$,draw] (v_{m-1}) at  (7,0.5) {};
\node[label=right:$j$,draw] (j) at  (9,-1) {};

\node[label=below:$w_1$,draw] (w_1) at  (1,-2.5) {};
\node[label=below:$w_2$,draw] (w_2) at  (3,-2.5) {};
\node[label=below:$Q$] at  (0,-1.75) {};
\node[label=below:$w_{s'}'$,draw] (w_{s'}'clone) at  (5,-2.5) {};
\node[label=below:$w_{l}$,draw] (w_{l-1}) at  (7,-2.5) {};

\node[label=below:$w_1'$,draw] (w_1') at  (7.5,-1) {};
\node[label=above:$R$] at (6.75,-1.15) {}; 
\node[label=below:$w_2'$,draw] (w_2') at  (6,-1) {};
\node[label=below:$w_s'$,draw] (w_s') at  (4,-1) {};
\node[label=below:$w_{s'}'$,draw] (w_{s'}') at  (2.5,-1) {};

\node[label=below:$w_{\alpha}'$,draw] (w_{alpha-1}') at  (0.5,-1) {};

\draw[edge]  (i) to (v_1);
\draw[edge]  (v_1) to (v_2);
\draw[edge]  (v_{m-1}) to (j);
\draw[edge]  (i) to (w_1);
\draw[edge]  (w_1) to (w_2);
\draw[edge]  (w_{l-1}) to (j);

\draw[edge]  (w_{alpha-1}') to (i);
\draw[edge]  (w_1') to (w_2');
\draw[edge]  (j) to (w_1');

\begin{scope}[dashed]
\draw[edge]  (w_2) to (w_{s'}'clone);
\draw[edge]  (w_{s'}'clone) to (w_{l-1});
\draw[edge]  (v_2) to (w_s'clone);
\draw[edge]  (w_s'clone) to (v_{m-1});
\draw  (w_s'clone) to (w_s'); 
\draw  (w_{s'}'clone) to (w_{s'}');
\draw[edge]  (w_2') to (w_s');
\draw[edge]  (w_s') to (w_{s'}');
\draw[edge]  (w_{s'}') to (w_{alpha-1}');
\end{scope}
\end{tikzpicture}}
\caption{}\label{F:unique_path:cactus}
\end{figure}

Suppose $P$ and $Q$ have some internal vertices in common. Let $v_s$ be the first vertex of $P$ in common with $Q$. This means there are two internally vertex disjoint directed paths from $i$ to $v_s$. By a similar argument as above we get a contradiction. Hence, the directed path from $i$ to $j$ is unique.    
\end{proof}

\begin{lem}\label{L:cact:partition}
Let $V:=\{1,\dotsc,n\}$ and $G= (V,E)$ be a directed cactus graph. Suppose $(i,j) \in E$. If both $i$ and $j$ have degree greater than one, then $V$ can be partitioned into three disjoint sets
\begin{enumerate}
    \item[\rm(a)] $\{i,j\}$
    \item[\rm(b)] $V_j(i \rightarrow)$
    \item[\rm(c)] $V_i(j \rightarrow)$,
\end{enumerate}
where $V_\nu(\delta \rightarrow)= \{k \in V \smallsetminus \{\delta,\nu\}: \exists$ a directed path from $\delta$ to $k$ which does not pass through $\nu\}$ (see \cref{F:cactoid:partition}).

\begin{figure}[tbhp]
\centering
\begin{tikzpicture}[shorten >=1pt, auto, node distance=3cm, ultra thick,
   node_style/.style={circle,draw=black,fill=white !20!,font=\sffamily\Large\bfseries},
   edge_style/.style={draw=black, ultra thick}]
\node[label=below:$i$,draw] (1) at  (1,0) {};
\node[label=below:$ $,draw] (i_1) at  (-1, 1) {};
\node[label=left:$V_j(i \rightarrow)$,draw] (i_2) at  (-1.5, 0.5) {};
\node[label=below:$ $,draw] (i_3) at  (-1.5, -0.5) {};
\node[label=below:$ $,draw] (i_4) at  (-1, -1) {};

\node[label=below:$j$,draw] (n) at  (5,0) {};
\node[label=below:$ $,draw] (j_1) at  (7, 1) {};
\node[label=right:$V_i(j \rightarrow)$,draw] (j_2) at  (7.5, 0.5) {};
\node[label=below:$ $,draw] (j_3) at  (7.5, -0.5) {};
\node[label=below:$ $,draw] (j_4) at  (7, -1) {};

\node[label=above:$ $,draw] (w_{1}) at  (6.5,1.5) {};

\draw[edge]  (1) to (n);
\draw[edge]  (w_{1}) to (1);

\begin{scope}[dashed]
\draw[edge]  (1) to (i_1);
\draw[edge]  (i_2) to (1);
\draw[edge]  (i_3) to (1);
\draw[edge]  (1) to (i_4);
\draw[edge]  (n) to (j_2);
\draw[edge]  (j_1) to (n);
\draw[edge]  (j_4) to (n);
\draw[edge]  (n) to (j_3);
\draw[edge]  (n) to (w_{1});
\end{scope}
\end{tikzpicture}
\caption{Partition of a directed cactus graph.}
\label{F:cactoid:partition}
\end{figure}
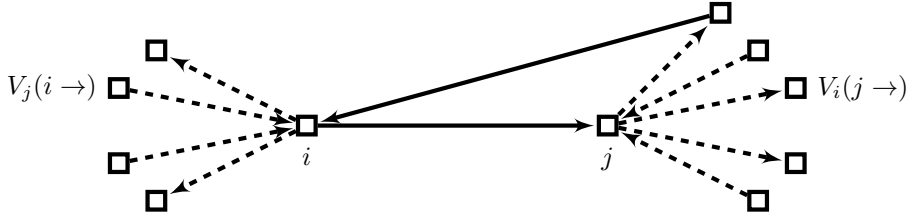
\end{lem}

\begin{proof}
As $i$ and $j$ have degree greater than one, $V_j(i \rightarrow)$ and $V_i(j \rightarrow)$ are non empty.
By definition, 
$$\{i,j\} \cap V_j(i \rightarrow) = \emptyset~\text{and}~ \{i,j\} \cap V_i(j \rightarrow) = \emptyset.$$
It remains to see that $V_j(i \rightarrow) \cap V_i(j \rightarrow) = \emptyset$. 
If possible, let $k \in V_j(i \rightarrow) \cap V_i(j \rightarrow)$. Let $P: i \rightarrow v_1 \rightarrow v_2 \rightarrow \cdots \rightarrow v_{m} \rightarrow k$ be a directed path from $i$ to $k$ which does not pass through $j$ and $R: j \rightarrow w_1 \rightarrow w_2 \rightarrow \cdots \rightarrow w_{\alpha} \rightarrow k$ be a directed path from $j$ to $k$ which does not pass through $i$ (see \cref{F:res<dist:cactoid:P:R}). Since $v_1 \neq j$, the edges $(i,j)$ and $(i,v_1)$ are not same. Hence, there are two different directed paths $P$ and $i \rightarrow j \rightarrow w_1 \rightarrow w_2 \rightarrow \cdots \rightarrow w_{\alpha} \rightarrow k$ from $i$ to $k$. This contradicts \cref{unique_path:cactus}. Thus, $V_j(i \rightarrow) \cap V_i(j \rightarrow) = \emptyset$.

\begin{figure}[tbhp]
\begin{center}
\begin{tikzpicture}[shorten >=1pt, auto, node distance=3cm, ultra thick,
   node_style/.style={circle,draw=black,fill=white !20!,font=\sffamily\Large\bfseries},
   edge_style/.style={draw=black, ultra thick}]
\node[label=above:$i$,draw] (1) at  (0,0) {};
\node[label=left:$P$] (1) at  (0,0) {};
\node[label=above:$v_1$,draw] (v_1) at  (1.5,-0.25) {};
\node[label=above:$v_2$,draw] (v_2) at  (3,-0.5) {};
\node[label=above:$v_{m}$,draw] (v_{m-1}) at  (4.5,-0.75) {};
\node[label=right:$k$,draw] (v_m) at  (6,-1) {};
\node[label=below:$j$,draw] (n) at  (0,-2) {};
\node[label=left:$R$] (n) at  (0,-2) {};
\node[label=below:$w_{\alpha}$,draw] (w_{l-1}) at  (4.5,-1.25) {};
\node[label=below:$w_{2}$,draw] (w_{2}') at  (3,-1.5) {};
\node[label=below:$w_{1}$,draw] (w_{1}') at  (1.5,-1.75) {};

\draw[edge]  (1) to (v_1);
\draw[edge]  (v_1) to (v_2);
\draw[edge]  (v_{m-1}) to (v_m);
\draw[edge]  (1) to (n);
\draw[edge]  (n) to (w_{1}');
\draw[edge]  (w_{1}') to (w_{2}');
\draw[edge]  (w_{l-1}) to (v_m);

\begin{scope}[dashed]
\draw[edge]  (v_2) to (v_{m-1});
\draw[edge]  (w_{2}') to (w_{l-1});
\end{scope}
\end{tikzpicture}
\caption{Directed paths $P$ and $R$.} \label{F:res<dist:cactoid:P:R}
\end{center}
\end{figure}
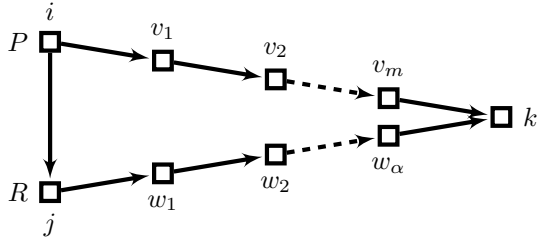 

Let $k \in V$ be such that $k \notin \{i,j\}$ and $k \notin V_j(i \rightarrow)$. Since $G$ is strongly connected there exists a directed path, say $P$ from $i$ to $k$. However, $P$ must pass through $j$. Thus, the part of $P$ between vertices $j$ and $k$ is a directed path from $j$ to $k$ which does not pass through $i$. So, $k \in V_i(j \rightarrow)$. Hence $V = \{i,j\} \cup V_j(i \rightarrow) \cup V_i(j \rightarrow)$ is a disjoint partition of $V$.
\end{proof}
For a subgraph $\widetilde{G}$ of $G$, we use $V(\widetilde{G})$ to denote the vertex set of $\widetilde{G}$. We now prove our main result.

\begin{thm}\label{res<dist,cactoid}
Let $G= (V,E)$ be a directed cactus graph with $V=\{1,2,\ldots,n\}$. If $R=(r_{ij})$ and $D=(d_{ij})$ are the resistance and distance matrices of $G$, respectively, then $r_{ij} \leq d_{ij}$ for each $i,j \in \{1,2,\ldots,n\}$. 
\end{thm}

\begin{proof}
By triangle inequality, it suffices to show that if $(i,j) \in E$, then $r_{ij} \leq 1$. Let $(i,j) \in E$. In view of \cref{resleq1fordeg1}, it suffices to show this inequality
when both $i$ and $j$ have degree greater than one. Without loss of generality, assume $i=1$ and $j=n$. By \cref{L:cact:partition}, the vertex set $V$ can be partitioned into three disjoint sets
\begin{enumerate}
    \item[(a)] $\{1,n\}$
    \item[(b)] $V_n(1 \rightarrow)$
    \item[(c)] $V_1(n \rightarrow)$.
\end{enumerate}  Let $L$ be the Laplacian matrix of $G$. From \cref{pinv}, we have 
\begin{equation}\label{Eres<dist:cactoid:pinv}
L^{\dag} = \left[
\begin{array}{cccc}
B^{-1} - \frac{1}{\displaystyle n} e e' B^{-1} - \frac{1}{\displaystyle n} B^{-1} ee' & - \frac{1}{\displaystyle n} B^{-1} e \\ \\
-\frac{1}{\displaystyle n} e' B^{-1} & 0 \\ 
\end{array}
\right] + \frac{e' B^{-1}e}{\displaystyle n^2} \1\1'.    
\end{equation}
where $B = \det(L[\{n\}^c,\{n\}^c])$. Let $C = B^{-1}$, $C = (c_{ij})$, $x = Ce$ and $y = C'e$. Note that $C$ is a non-negative matrix. Using \cref{Eres<dist:cactoid:pinv}, we have
\begin{equation}\label{Eres<dist,cactoid,r1n}
\begin{aligned}
    r_{1n} &= l^{\dag}_{11}+l^{\dag}_{nn}-2l^{\dag}_{1n} \\ &= c_{11}-\frac{1}{n}y_1-\frac{1}{n}x_1+\frac{2}{n}x_1 \\ &= c_{11}-\frac{1}{n}(y_1-x_1).
\end{aligned}
\end{equation}
As in proof of \cref{resleq1fordeg1}, it suffices to show that $x_1 \leq y_1$. Let $k\in \{2,3,\ldots,n-1\}$. Then by \cref{Eresleq1fordeg1:deg1=1:c1k=n(F)}
\begin{equation}\label{Eres<dist:cactoid:c1k=n(F)}
     c_{1k} = \frac{\#(F[\{n\rightarrow\},\{1\rightarrow,k\}])}{\kappa(G)},
\end{equation}
where $\#(F[\{n\rightarrow\},\{1\rightarrow,k\}])$ is the number of spanning forests of $G$ such that (i) $F$ contains exactly $2$ trees, (ii) each tree in $F$ exactly contains either $n$ or both $k$ and $1$, and (iii) vertices $n$ and $1$ are the roots of the respective trees containing them. We shall show that for each $k \in V_n(1 \rightarrow)$, such a forest $F$ exists and is unique. Fix $k \in V_n(1 \rightarrow)$.

Existence:~Since $\kappa(G) = \det(B) \neq 0$, there is a spanning tree $T$ of $G$ rooted at $1$. Since the edge $(1,n)$ is the only directed path from $1$ to $n$ in $G$, it must be an edge in $T$. By removing the edge $(1,n)$ from $T$, we obtain a spanning forest $\overline{F}$ containing exactly two trees $\overline{T_1}$ and $\overline{T_n}$ rooted at $1$ and $n$, respectively. It remains to show that $k \in V(T_1)$. In order to do this, we prove that 
\[V(\overline{T_1}) = V_n(1 \rightarrow) \cup \{1\}~~\mbox{and}~~ V(\overline{T_n}) = V_1(n \rightarrow)\cup \{n\}.\] Let $v \in V(\overline{T_1}) \smallsetminus \{1\}$. Then there is a directed path from $1$ to $v$ in $G$ which does not pass through $n$. This implies $V(\overline{T_1}) \subseteq V_n(1 \rightarrow)\cup \{1\}$. Similarly, if $w \in V(\overline{T_n}) \backslash \{n\}$ then there is a directed path from $n$ to $w$ in $G$ which does not pass through $1$ and so $V(\overline{T_n}) \subseteq V_1(n \rightarrow)\cup \{n\}$. Let $v \in V_n(1 \rightarrow)$ be such that $v \notin V(\overline{T_1})$. Since $V = V(\overline{T_1}) \cup V(\overline{T_n})$, we have $v \in V(\overline{T_n}) \subseteq V_1(n \rightarrow)\cup \{n\}$. Thus,
$$v \in V_n(1 \rightarrow) \cap V_1(n \rightarrow),$$ which is a contradiction. Hence $V(\overline{T_1}) = V_n(1 \rightarrow) \cup \{1\}$. Similarly, $V(\overline{T_n}) = V_1(n \rightarrow)\cup \{n\}$. Thus, $k \in V(\overline{T_1})$ and so $$\#(F[\{n\rightarrow\},\{1\rightarrow,k\}]) \neq 0.$$
Hence, a forest with the required properties exists.

Uniqueness: Let $\widetilde{F}$ be another forest other than $\overline{F}$. Suppose $\widetilde{T_1}$ and $\widetilde{T_n}$ be the trees in $\widetilde{F}$ rooted at $1$ and $n$, respectively. It can be easily seen $$V(\widetilde{T_1}) = V_n(1 \rightarrow) \cup \{1\} = V(\overline{T_1})~~\mbox{and}~~V(\widetilde{T_n}) = V_1(n \rightarrow)\cup \{n\} = V(\overline{T_n}).$$ If the trees $\overline{T_1}$ and $\widetilde{T_1}$ are not identical, then there will be a vertex $v \in V(\overline{T_1})$ such that there are two different directed paths from $1$ to $v$. This contradicts \cref{unique_path:cactus}. So, the trees $\overline{T_1}$ and $\widetilde{T_1}$ are identical. The same happens with $\overline{T_n}$ and $\widetilde{T_n}$. Thus, the forest  $\overline{F}$ is unique and so 
\begin{equation}\label{E:nF[n:1tok]=1}
    \#(F[\{n\rightarrow\},\{1\rightarrow,k\}])=1
\end{equation}for each $k \in V_n(1 \rightarrow)$.

Since for every $k \notin V_n(1 \rightarrow)$, there is no directed path from $1$ to $k$ that does not pass through $n$, $\#(F[\{n\rightarrow\},\{1\rightarrow,k\}]) = 0$. From \cref{Eres<dist:cactoid:c1k=n(F)} and \cref{E:nF[n:1tok]=1}, for each $k\in \{2,3,\ldots,n-1\}$

\begin{equation}\label{Eres<dist:cactoid:c1k=1:0} 
    c_{1k} = \begin{cases}
    \frac{1}{\displaystyle \kappa(G)} &~ \text{if}~ k \in V_n(1 \rightarrow) \\
    0 &~ \text{otherwise.}
    \end{cases}
\end{equation}

Let $V_n(\rightarrow 1)= \{k \in V \smallsetminus \{1,n\}: \exists$ directed path from $k$ to $1$ which does not pass through $n \}$. Now, we shall show that $V_n(1 \rightarrow) \subset V_n(\rightarrow 1)$. Let $k \in V_n(1 \rightarrow)$ and $P: 1 \rightarrow v_1 \rightarrow v_2 \rightarrow \cdots \rightarrow v_{m} \rightarrow k$ be a directed path from $1$ to $k$ which does not pass through $n$. If possible, assume $k \notin V_n(\rightarrow 1)$ i.e. every directed path from $k$ to $1$ contains the vertex $n$. Since $G$ is strongly connected, there is at least one such path say $Q: k \rightarrow w_1 \rightarrow w_2 \rightarrow \cdots \rightarrow w_{l} \rightarrow n \rightarrow w_1' \rightarrow \cdots \rightarrow w_{\alpha}' \rightarrow 1$  (see \cref{F:res<dist:cactoid:P:Q}). Since $v_1 \neq n$, the edges $(1,n)$ and $(1,v_1)$ are not same. Hence, there are two different directed paths $1 \rightarrow n$ and $1 \rightarrow v_1 \rightarrow v_2 \rightarrow \cdots \rightarrow v_{m} \rightarrow k \rightarrow w_1 \rightarrow w_2 \rightarrow \cdots \rightarrow w_{l} \rightarrow n$ from $1$ to $n$. This contradicts \cref{unique_path:cactus}. Hence $V_n(1 \rightarrow) \subset V_n(\rightarrow 1)$. 

\begin{figure}[tbhp]
\centering
\begin{tikzpicture}[shorten >=1pt, auto, node distance=3cm, ultra thick,
   node_style/.style={circle,draw=black,fill=white !20!,font=\sffamily\Large\bfseries},
   edge_style/.style={draw=black, ultra thick}]
\node[label=left:$1$,draw] (1) at  (0,0) {};
\node[label=below:$P$] at  (0.75,0) {};
\node[label=below:$v_1$,draw] (v_1) at  (1.5,0) {};
\node[label=below:$v_2$,draw] (v_2) at  (3,0) {};
\node[label=below:$v_{m}$,draw] (v_{m-1}) at  (4.5,0) {};
\node[label=right:$k$,draw] (v_m) at  (6,0) {};
\node[label=right:$Q$] at  (6,0.75) {};
\node[label=above:$n$,draw] (n) at  (1.5,1.5) {};
\node[label=above:$w_1'$,draw] (w_1') at  (2.25,3) {};
\node[label=above:$ $,draw] (w_2') at  (0.75,3) {};
\node[label=right:$ $,draw] (C) at  (0,2) {};
\node[label=left:$w_{\alpha}'$,draw] (w_{l-1}') at  (0,1) {};

\node[label=above:$w_{l}$,draw] (w_{l-1}) at  (3,1.5) {};
\node[label=above:$w_{1}$,draw] (w_{1}) at  (6,1.5) {};
\node[label=above:$w_{2}$,draw] (w_{2}) at  (4.5,1.5) {};

\draw[edge]  (1) to (v_1);
\draw[edge]  (v_1) to (v_2);
\draw[edge]  (v_{m-1}) to (v_m);
\draw[edge]  (v_m) to (w_{1});
\draw[edge]  (w_{1}) to (w_{2});
\draw[edge]  (w_{l-1}) to (n);
\draw[edge]  (1) to (n);
\draw[edge]  (n) to (w_1');
\draw[edge]  (w_1') to (w_2');
\draw[edge]  (w_{l-1}') to (1);

\begin{scope}[dashed]
\draw[edge]  (v_2) to (v_{m-1});
\draw[edge]  (w_{2}) to (w_{l-1});
\draw[edge]  (w_2') to (C);
\draw[edge]  (C) to (w_{l-1}');
\end{scope}
\end{tikzpicture}
\caption{Directed paths $P$ and $Q$ } \label{F:res<dist:cactoid:P:Q} 
\end{figure} 
If $k\in \{2,3,\ldots,n-1\}$, then 
    \begin{equation}\label{Eres<dist:cactoid:ck1=det}
    \begin{aligned}
        c_{k1} &= \frac{(-1)^{1+k}}{\det(B)} \det(B[\{1\}^c,\{k\}^c]) \\ &= \frac{(-1)^{1+k}}{\det(L[\{n\}^c,\{n\}^c])} \det(L[\{n,1\}^c,\{n,k\}^c]).
    \end{aligned}
    \end{equation}

Using \cref{kappa} and \cref{mtt:2vert:rem}(b) in \cref{Eres<dist:cactoid:ck1=det}, we get
\begin{equation}\label{Eres<dist:cactoid:ck1=n(F)}
\begin{aligned}
     c_{k1} &= \frac{\#(F[\{n\rightarrow\},\{k\rightarrow,1\}])}{\kappa(G)} 
\end{aligned}
\end{equation}
where $\#(F[\{n\rightarrow\},\{k\rightarrow,1\}])$ is the number of spanning forests of $G$ such that (i) $F$ contains exactly $2$ trees, (ii) each tree in $F$ exactly contains either $n$ or both  $k$ and $1$, and (iii) vertices $n$ and $k$ are the roots of the respective trees containing them. We shall show that for each $k \in V_n(1 \rightarrow)$, $\#(F[\{n\rightarrow\},\{k\rightarrow,1\}]) \geq 1$.

Consider the induced subgraph $\widetilde{G}$ of $G$ with vertex set $V(\widetilde{G}) = V_n(1\rightarrow) \cup \{1\}$. Note that for each vertex $x \in V(\widetilde{G})$ there is a directed path from $1$ to $x$ in $\widetilde{G}$. Since $V(\widetilde{G}) \smallsetminus \{1\} \subset V_n(\rightarrow 1)$,  for every $x \in V(\widetilde{G})$ there is a directed path from $x$ to $1$ in $G$ which does not pass through $n$. For fixed $x \in V(\widetilde{G})\smallsetminus \{1\}$, let $P: x \rightarrow v_1 \rightarrow v_2 \rightarrow \cdots \rightarrow v_{m} \rightarrow 1$ be one such path. We claim that each internal vertex of $P$ is a vertex in $\widetilde{G}$. If possible, let $v_s$ be an internal vertex of $P$ which is not in $V(\widetilde{G})$. Since $v_s \notin V_n(1 \rightarrow)$, there is a directed path say $Q: 1 \rightarrow w_1 \rightarrow w_2 \rightarrow \cdots \rightarrow w_{l} \rightarrow n \rightarrow w_{l+1} \rightarrow \cdots \rightarrow w_{l'} \rightarrow v_s$ from $1$ to $v_s$ which passes through $n$ (see \cref{F:Gtilde:sc}). Also there will be a directed path say $R: 1 \rightarrow w'_1 \rightarrow w'_2 \rightarrow \cdots \rightarrow w'_{\alpha} \rightarrow x$ from $1$ to $x$ which does not pass through $n$. Note that, there are two different paths $Q$ and $1 \rightarrow w'_1 \rightarrow w'_2 \rightarrow \cdots \rightarrow w'_{\alpha} \rightarrow x \rightarrow v_1 \rightarrow v_2 \rightarrow \cdots \rightarrow v_{s}$ from $1$ to $v_s$. This contradicts \cref{F:unique_path:cactus}. Hence each internal vertex of $P$ is a vertex in $\widetilde{G}$. Thus, $P$ is a directed path from $x$ to $1$ in $\widetilde{G}$. 

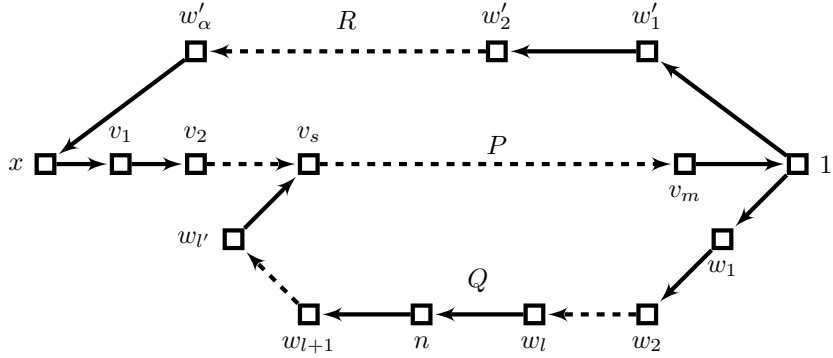
\begin{figure}[tbhp]
\centering
\begin{tikzpicture}[shorten >=1pt, auto, node distance=3cm, ultra thick,
   node_style/.style={circle,draw=black,fill=white !20!,font=\sffamily\Large\bfseries},
   edge_style/.style={draw=black, ultra thick}]
\node[label=left:$x$,draw] (x) at  (-1,-1) {};
\node[label=above:$R$] at  (3,0.5) {};
\node[label=above:$w_{\alpha}'$,draw] (w_{alpha}') at  (1,0.5) {};
\node[label=above:$w_2'$,draw] (w_2') at  (5,0.5) {};
\node[label=above:$w_1'$,draw] (w_1') at  (7,0.5) {};
\node[label=right:$1$,draw] (1) at  (9,-1) {};

\node[label=below:$n$,draw] (n) at  (4,-3) {};
\node[label=below:$w_1$,draw] (w_1) at  (8,-2) {};
\node[label=below:$w_2$,draw] (w_2) at  (7,-3) {};
\node[label=above:$Q$] at  (4.75,-3) {};
\node[label=below:$w_l$,draw] (w_l) at  (5.5,-3) {};
\node[label=left:$w_{l'}$,draw] (w_{l'}) at  (1.5,-2) {};
\node[label=below:$w_{l+1}$,draw] (w_{l+1}) at  (2.5,-3) {};

\node[label=below:$v_m$,draw] (v_m) at  (7.5,-1) {};
\node[label=above:$P$] at (5,-1.15) {}; 
\node[label=above:$v_s$,draw] (v_s) at  (2.5,-1) {};
\node[label=above:$v_2$,draw] (v_2) at  (1,-1) {};
\node[label=above:$v_1$,draw] (v_1) at  (0,-1) {};

\draw[edge]  (x) to (v_1);
\draw[edge]  (v_1) to (v_2);
\draw[edge]  (v_m) to (1);
\draw[edge]  (1) to (w_1);
\draw[edge]  (w_1) to (w_2);
\draw[edge]  (w_l) to (n);
\draw[edge]  (n) to (w_{l+1});
\draw[edge]  (w_{l'}) to (v_s);

\draw[edge]  (w_{alpha}') to (x);
\draw[edge]  (w_1') to (w_2');
\draw[edge]  (1) to (w_1');

\begin{scope}[dashed]
\draw[edge]  (w_2) to (w_l);
\draw[edge]  (v_2) to (v_s);
\draw[edge]  (v_s) to (v_m);
\draw[edge]  (w_{l+1}) to (w_{l'});
\draw[edge]  (w_2') to (w_{alpha}');
\end{scope}
\end{tikzpicture}
\caption{Directed paths $P$,$Q$ and $R$ in $\widetilde{G}$.} \label{F:Gtilde:sc} 
\end{figure} 

This means $\widetilde{G}$ is a strongly connected digraph. Let $\widetilde{L}$ be the Laplacian matrix of $\widetilde{G}$. Then $\rank(\widetilde{L}) = |V(\widetilde{G})|-1$. Let $k \in V_n(1\rightarrow)$. Then $\kappa(\widetilde{G},k) = \det(\widetilde{L}[\{k\}^c,\{k\}^c]) \neq 0$. Hence there exists an oriented spanning tree of $\widetilde{G}$ rooted at $k$. Let $\widetilde{T}$ be a spanning tree of $\widetilde{G}$ rooted at $k$ and $\overline{F}$ be the spanning forest of $G$ with trees $\overline{T_1}$ and $\overline{T_n}$ obtained as before. Let $F'$ be the forest consisting of trees $\widetilde{T}$ and $\overline{T_n}$. Since $\widetilde{T}$ and $\overline{T_n}$ are rooted at $k$ and $n$, respectively, and $V(\widetilde{T}) = V_n(1\rightarrow) \cup \{1\}$ and $V(\overline{T_n}) = V_1(n \rightarrow) \cup \{n\}$, it follows that $F'$  is a required spanning forest. Hence for each $k \in V_n(1\rightarrow)$, we have $\#(F[\{n\rightarrow\},\{k\rightarrow,1\}]) \geq 1$. From \cref{Eres<dist:cactoid:ck1=n(F)}, we have
\begin{equation}\label{Eres<dist:cactoid:ck1>1}
c_{k1}  \geq \frac{1}{\kappa(G)},~~~\text{whenever}~ k \in  V_n(1\rightarrow).
\end{equation}
Since $C$ is a non-negative matrix, from \cref{Eres<dist:cactoid:c1k=1:0} and \cref{Eres<dist:cactoid:ck1>1}, we have
\begin{equation}
\begin{aligned}
x _1 &= \sum_{k=1}^{n-1} c_{1k} \\ &= c_{11}+ \sum_{k \in  V_n(1\rightarrow)}c_{1k} \\ &= c_{11}+ \sum_{k \in  V_n(1\rightarrow)}\frac{1}{\kappa(G)} \\ &\leq   c_{11}+ \sum_{k \in  V_n(1\rightarrow)}c_{k1}   \leq \sum_{k=1}^{n-1} c_{k1} = y_1.
\end{aligned}
\end{equation}
Hence, $r_{1n} \leq 1$. This completes the proof.
\end{proof}
We complete the paper with the following example.
\begin{example}\rm
Consider the strongly connected and directed cactus graph $G$.
\begin{figure}[tbhp]
\centering
\begin{tikzpicture}[shorten >=1pt, auto, node distance=3cm, ultra thick,
   node_style/.style={circle,draw=black,fill=white !20!,font=\sffamily\Large\bfseries},
   edge_style/.style={draw=black, ultra thick}]
\node[label=above:$1$,draw] (1) at  (1.75,0) {};
\node[label=above:$2$,draw] (2) at  (3.5,0) {};
\node[label=below:$3$,draw] (3) at  (3.5,-2) {};
\node[label=above:$6$,draw] (6) at  (0,0) {};
\node[label=below:$4$,draw] (4) at  (0,-2) {};
\node[label=left:$5$,draw] (5) at  (-1.5,-1) {};
\node[label=below:$7$,draw] (7) at  (1.75,-2) {}; 
\draw[edge]  (1) to (2);
\draw[edge]  (2) to (3);
\draw[edge]  (3) to (1);
\draw[edge]  (1) to (4);
\draw[edge]  (4) to (5);
\draw[edge]  (5) to (6);
\draw[edge]  (6) to (1);
\draw[edge]  (1.65,-0.38) to (1.65,-1.7);
\draw[edge]  (1.85,-1.62) to (1.85,-0.35);
\end{tikzpicture}
\caption{The directed cactus graph $G$.} \label{eg:cactus2}
\end{figure}
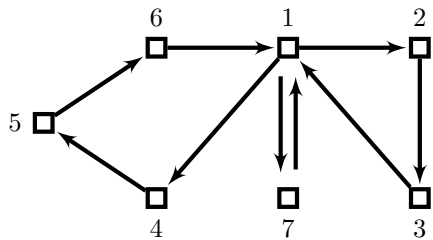 
The resistance and distance matrices of $G$ are:
\begin{small}
\begin{equation*}\label{eg,res&dis,cactus}
R = 
\left[
{\begin{array}{rrrrrrr}
0 & \frac{6}{7} & \frac{8}{7} & \frac{5}{7} & 1 & \frac{9}{7} & 1 \\
\frac{8}{7} & 0 & \frac{2}{7} & \frac{13}{7} & \frac{15}{7} & \frac{17}{7} & \frac{15}{7} \\
\frac{6}{7} & \frac{12}{7} & 0 & \frac{11}{7} & \frac{13}{7} & \frac{15}{7} & \frac{13}{7} \\
\frac{9}{7} & \frac{15}{7} & \frac{17}{7} & 0 & \frac{2}{7} & \frac{4}{7} & \frac{16}{7} \\
1 & \frac{13}{7} & \frac{15}{7} & \frac{12}{7} & 0 & \frac{2}{7} & 2 \\
\frac{5}{7} & \frac{11}{7} & \frac{13}{7} & \frac{10}{7} & \frac{12}{7} & 0 & \frac{12}{7} \\
1 & \frac{13}{7} & \frac{15}{7} & \frac{12}{7} & 2 & \frac{16}{7} & 0
\end{array}}
\right]~~\mbox{and}~~ 
D = 
\left[
{\begin{array}{rrrrrrr}
0 & 1 & 2 & 1 & 2 & 3 & 1 \\
2 & 0 & 1 & 3 & 4 & 5 & 3 \\
1 & 2 & 0 & 2 & 3 & 4 & 2 \\
3 & 4 & 5 & 0 & 1 & 2 & 4 \\
2 & 3 & 4 & 3 & 0 & 1 & 3 \\
1 & 2 & 3 & 2 & 3 & 0 & 2 \\
1 & 2 & 3 & 2 & 3 & 4 & 0
\end{array}}
\right].
\end{equation*} 
\end{small}
It can be seen that for each pair of vertices, the resistance is less than the shortest distance.
\end{example}


\bibliography{references}
\end{document}